\documentclass[12pt]{amsart}

\usepackage{amsfonts}
\usepackage{amssymb}
\usepackage{graphicx}
\usepackage{enumerate}
\usepackage{pgf,tikz}
\usepackage[letterpaper, left=2.5cm, right=2.5cm, top=2.5cm,
bottom=2.5cm,dvips]{geometry}
\usepackage{amsmath,amssymb,amsthm}
\usepackage{mathrsfs}
\usepackage{mathabx}\changenotsign
\usepackage{dsfont}

\usepackage[backref]{hyperref}
\hypersetup{
	colorlinks,
	linkcolor={blue!60!black},
	citecolor={green!60!black},
	urlcolor={red!60!black}
}

\newtheorem{theorem}{Theorem}
\theoremstyle{plain}

\newtheorem{case}{Case}

\newtheorem{conjecture}{Conjecture}

\newtheorem{lemma}{Lemma}

\numberwithin{equation}{section}

\DeclareMathOperator{\modulo}{mod}

\begin{document}
	\title[Reflections on the Erd\H {o}s Discrepancy Problem]{Reflections on the Erd\H {o}s Discrepancy Problem}
	
	\author{Bart\l omiej Bosek}
	\address{Theoretical Computer Science Department, Faculty of Mathematics and Computer Science, Jagiellonian
		University, 30-348 Krak\'{o}w, Poland}
	\email{bosek@tcs.uj.edu.pl}
	\author{Jaros\l aw Grytczuk}
	\address{Faculty of Mathematics and Information Science, Warsaw University
		of Technology, 00-662 Warsaw, Poland}
	\email{j.grytczuk@mini.pw.edu.pl, b.pilat@mini.pw.edu.pl}
	\thanks{Research supported by the National Science Center of Poland, grant 2015/17/B/ST1/02660}

\begin{abstract}
We consider some coloring issues related to the famous Erd\H {o}s Discrepancy Problem. A set of the form $A_{s,k}=\{s,2s,\dots,ks\}$, with $s,k\in \mathbb{N}$, is called a \emph{homogeneous arithmetic progression}. We prove that for every fixed $k$ there exists a $2$-coloring of $\mathbb N$ such that every set $A_{s,k}$ is \emph{perfectly balanced} (the numbers of red and blue elements in the set $A_{s,k}$ differ by at most one). This prompts reflection on various restricted versions of Erd\H {o}s' problem, obtained by imposing diverse confinements on parameters $s,k$. In a slightly different direction, we discuss a \emph{majority} variant of the problem, in which each set $A_{s,k}$ should have an excess of elements colored differently than the first element in the set. This problem leads, unexpectedly, to some deep questions concerning completely multiplicative functions with values in $\{+1,-1\}$. In particular, whether there is such a function with partial sums bounded from above.
\end{abstract}

\maketitle

\section{Introduction}
For a number $h\in \mathbb{N}$, a red-blue coloring of a finite set $A$ is said to be \emph{$h$-balanced} if the numbers of red and blue elements in $A$ differ by at most $h$. If $h=1$, then we call it \emph{perfectly balanced}. For positive integers $s,k\in \mathbb N$, we denote by $A_{s,k}=\{s,2s,\dots,ks\}$ the homogeneous arithmetic progression of length $k$ and step $s$.

In 1932 Erd\H {o}s \cite{Erdos} asked if there exists a constant $h$ and a red-blue coloring of $\mathbb{N}$ such that \emph{every} homogeneous arithmetic progression is $h$-balanced. The property in question seems unbelievable, and in fact, he expressed a guess that there is no such constant. This was confirmed only recently by Tao \cite{Tao}, with a support of collective efforts in a Polymath Project \cite{Polymath}.

We prove in this note that for every fixed $k\in \mathbb{N}$ there is a red-blue coloring of $\mathbb N$ which is perfectly balanced on all sets $A_{s,k}$.

\begin{theorem}\label{Theorem Main}
	For every fixed $k\in\mathbb N$ there exists a red-blue coloring of $\mathbb N$ such that every homogeneous arithmetic progression $A_{s,k}$ is perfectly balanced.
\end{theorem}

The proof uses completely multiplicative functions with values in the set $\{+1,-1\}$ and some estimates for primes in arithmetic progressions. We give it in Section 2. In Section 3 we present another approach that did not appear to be successful, but leads to an intriguing open problem. In Section 4 we propose a new variant of the Erd\H {o}s Discrepancy Problem, motivated by the \emph{majority} coloring of graphs. This leads in turn to a question concerning completely multipliactive functions, resembling the famous conjecture of P\'{o}lya \cite{Polya}, concerning partial sums of the \emph{Liouville function} $\lambda(n)$. Finally, in Section 5, we briefly describe our initial motivation that led us to Theorem \ref{Theorem Main}, and discuss the related problem concerning \emph{rainbow} homogeneous arithmetic progressions of fixed length $k$.

\section{The proof}

Recall that an arithmetic function $f$ is \emph{completely multiplicative} if it satisfies $f(ab)=f(a)f(b)$ for every pair of positive integers $a,b\in \mathbb{N}$. Notice that this implies that $f(1)=1$. Since we will consider only functions with two values $\{+1,-1\}$ (corresponding to colors red and blue), we will call them shortly \emph{multiplicative colorings}.

We start with the following simple lemma.

\begin{lemma}\label{Lemma Multiplicative Balanced}
Let $k$ be a fixed positive integer. Suppose that $c$ is a multiplicative coloring of the set $\{1,2,\dots,k\}$ which is perfectly balanced. Then $c$ can be extended to a multiplicative coloring of $\mathbb N$ which is perfectly balanced on every set $A_{s,k}$.
\end{lemma}
\begin{proof}
Let $C=(c(1),c(2),\dots,c(k))$ be the initial color sequence satisfying the assumptions of the lemma. In particular, it satisfies
\begin{equation}
	c(1)+c(2)+\cdots +c(k) \in \{+1,-1,0\}.
\end{equation}

We extend the coloring $c$ to the whole of $\mathbb N$ in a natural way. First, if $p\geqslant k+1$ is a prime number, then we may chose for $c(p)$ any value from $\{+1,-1\}$. If $n=p_1p_2\cdots p_r$ is a product of primes $p_i$, then we compute $c(n)=c(p_1)c(p_2)\cdots c(p_r)$. So, the coloring $c$ is multiplicative. Consequently, the color sequence of every set $A_{s,k}$ satisfies 
\begin{equation}
(c(s),c(2s),\dots,c(ks)) =c(s)(c(1),c(2),\dots,c(k))=\pm C,
\end{equation}
and is therefore perfectly balanced.
\end{proof}

The next lemma comes from the paper by Borwein, Choi, and Coons \cite{BorweinCC}.

\begin{lemma}[Borwein, Choi, Coons, \cite{BorweinCC}]\label{Lemma Borwein}
	Let $b$ be a multiplicative coloring of $\mathbb{N}$ defined by $b(3)=+1$ and $b(p)\equiv p(\modulo 3)$ for all other primes $p$. Then for every $k\geqslant 1$, the sum $\Sigma_{i=1}^{k}b(i)$ equals the number of $1$'s in base $3$ expansion of $k$. In consequence, for every $k\geqslant 1$ we have
	\begin{equation}
0\leqslant\Sigma_{i=1}^{k}b(i)\leqslant \lceil \log_3 k\rceil +1.
	\end{equation}
\end{lemma}

We will also need some estimates on the number of primes of the form $3m+1$ in the interval $(N,2N)$. Recall that the Chebyshev function $\theta(x;3,1)$ is defined by 
\begin{equation}
\theta(x;3,1)=\sum_{p\equiv1(\modulo3),p\leqslant x}\log p.
\end{equation}
We will use the following result of McCurley \cite{McCurley}.

\begin{lemma}[McCurley \cite{McCurley}]\label{Lemma McCurley}
For $x\geqslant 17377$ we have
\begin{equation}
0.49x\leqslant \theta(x;3,1)\leqslant 0.51x.
\end{equation}
\end{lemma}
Using this lemma we get a useful lower bound for the number of primes of the form $3m+1$ between $x$ and $2x$.

\begin{lemma}\label{Lemma f(N)}
Let $f(x)$ denote the number of primes of the form $3m+1$ in the interval $(x,2x)$. Then for every $x\geqslant17377$ we have
\begin{equation}
f(x)\geqslant 0.47\frac{x}{\log(2x)}.
\end{equation}
\end{lemma}
\begin{proof}
	Let $f(x)=r$, and let $q_1,q_2,\dots, q_r$ be all the primes of the form $3m+1$ in the interval $(x,2x)$. Then we may write
	\begin{equation}
	\theta(2x;3,1)-\theta(x;3,1)=\sum_{i=1}^{r} \log q_i \leqslant \log ((2x)^r)=r\log(2x).
	\end{equation}
	Using Lemma \ref{Lemma McCurley} we get
	\begin{equation}
	\theta(2x;3,1)-\theta(x;3,1)\geqslant 0.98x-0.51x=0.47x.
	\end{equation}
	Hence, we obtain $r\geqslant \frac{0.47x}{\log(2x)}$ for all $x\geqslant 17377$.
\end{proof}

Now we are ready to prove Theorem \ref{Theorem Main}
\begin{proof}[Proof of Theorem \ref{Theorem Main}]

Let $k\in \mathbb{N}$ be fixed. By Lemma \ref{Lemma Multiplicative Balanced} it is sufficient to construct a multiplicative perfectly balanced coloring of $\{1,2,\dots,k\}$. For small values of $k$, say for $k\leqslant 10^6$, this can be done by computer. So, assume that $k\geqslant 17377$, and let $b$ be a multiplicative coloring form Lemma \ref{Lemma Borwein}.

We will switch the colors of prime numbers of the form $p=3m+1$ lying in the range $(\frac{k}{2},k)$, from plus to minus, so that the resulting coloring will be balanced. This will not affect multiplicativity. Moreover, there are sufficiently many such primes since by Lemma \ref{Lemma Borwein} and Lemma \ref{Lemma f(N)}, their number satisfies
\begin{equation}
f\left(\frac{k}{2}\right)\geqslant 0.47\frac{k}{2\log k}\geqslant \lceil \log_3{k} \rceil+1 \geqslant \sum_{i=1}^{k}b(i).
\end{equation}
This completes the proof.
\end{proof}

\section{Greedy muliplicative coloring}

To prove Theorem \ref{Theorem Main} we firstly considered a different approach proposed by Rejmer (personal communication). It is a simple algorithm producing a perfectly balanced multiplicative coloring of the set $\{1,2,\dots,k\}$ in a \emph{greedy} way.

Let $k$ be a fixed positive integer. We start with putting $c(1)=+1$. In each consecutive step we color the next integer so that the new coloring is perfectly balanced and multiplicative. So, in the second step we put $c(2)=-1$. For a more precise description, suppose that after $j-1$ steps, $j\geqslant 2$, we obtained a perfectly balanced multiplicative coloring $(c(1),c(2),\dots, c(j-1))$. In the next step we distinguish two cases.
\begin{case}[$j-1=2m$]\emph{In this case we must have $\sum_{i=1}^{2m}c(i)=0$. Thus, any choice for $c(j)$ will not destroy the perfect balance. If $j$ is composite, then $c(j)$ is determined by multiplicativity. If $j$ is prime, then we may put $c(j)=-1$.}
\end{case}
\begin{case}[$j-1=2m-1$]\emph{Then $j$ is even, so $c(j)$ is determined by multiplicativity. Since $\sum_{i=1}^{2m-1}c(i)=\pm 1$, we may have either $\sum_{i=1}^{2m}c(i)=0$ or $\sum_{i=1}^{2m}c(i)=\pm 2$. In the former case we do nothing. In the later case, we find the largest prime $p>\frac{j}{2}$ such that $c(p)$ has \textquotedblleft wrong\textquotedblright sign, and switch it. This makes the new coloring $(c(1),c(2),\dots ,c(j))$ perfectly balanced.}
\end{case}
Notice that by Bertrand's Postulate, there is always a prime between $\frac{j}{2}$ and $j$. However, it is not clear that there will always be a prime whose sign-switching would improve balance. For instance, in the $16$th step of the algorithm we get the following coloring:
$$
\begin{array}{|c|c|c|c|c|c|c|c|c|c|c|c|c|c|c|c|}
\hline
1&2&3&4&5&6&7&8&9&10&11&12&13&14&15&16 \\
\hline
+&-&-&+&-&+&-&-&+&+&+&-&-&+&+&+\\
\hline
\end{array}
$$
with $9$ pluses and $7$ minuses. To fix this imbalance, we go back to the first prime to the left, which is $p=13$. However, $c(13)=-1$, so switching this color into $c(13)=+1$ would only increase imbalance. Fortunately, for the next prime $p=11$ we have $c(11)=+1$, so we may switch it to $c(11)=-1$ and get in this way a perfectly balanced multiplicative coloring.

We do not known if this algorithm ever stops.

\begin{conjecture}\label{Conjecture Rejmer}
	Rejmer's algorithm runs ad infinitum.
\end{conjecture}

Rejmer made some computational experiments with his algorithm. In particular, he run it up to $10^9$ steps producing in this way perfectly balanced multiplicative colorings of $A_{1,k}$ for all $k\leqslant 10^9$. Notice that the first half terms of the last coloring will not be changed in the future. Thus, assuming validity of Conjecture \ref{Conjecture Rejmer}, the algorithm defines an intriguing recursive binary sequence $R(n)$ over $\{-1,+1\}$. Up to $n=40$ Rejmer's sequence coincides with the Liouville function $\lambda(n)$ (defined by $\lambda(p)=-1$ for all primes $p$), but $R(41)=+1$. The same happens for many other primes, in particular $R(97)=R(101)=+1$. One may suspect that there will be infinitely many primes $p$ with $R(p)=+1$.

\begin{conjecture}\label{Conjecture Rejmer 2}
There exist infinitely many primes $p$ for which $R(p)=+1$.
\end{conjecture}

\section{Majority version of the Erd\H {o}s Discrepancy Problem}

Let $h$ be a positive integer and let $c$ be a red-blue coloring of the set $A_{s,k}$. We say that $c$ is a \emph{majority} coloring if more than a half of the elements of $A_{s,k}-\{s\}$ have color different than the element $s$. Notice that a perfectly balanced coloring of $A_{s,k}$ satisfies the majority condition, while a majority coloring of $A_{s,k}$ can have all elements, except one, in the same color.

Is it possible that there is a red-blue coloring of $\mathbb{N}$ which satisfies the majority condition on \emph{every} homogeneous arithmetic progression $A_{s,k}$? Interpreting colors as numbers $\{+1,-1\}$, we may express the majority coloring of $A_{s,k}$ via the inequality:
\begin{equation}
c(a)\sum_{j=1}^{k}c(ja)\leqslant 0.
\end{equation}
So, the answer to the above question would be positive if we could find a completely multiplicative function $c$ satisfying the inequality
\begin{equation}
\sum_{j=1}^{k}c(j)\leqslant 0,
\end{equation}
for every $k\geqslant2$. A natural candidate for such negativity property is the Liouville function $\lambda(n)$. Actually in 1919 P\'{o}lya \cite{Polya} conjectured that $\sum_{i=1}^{n}\lambda(i)\leqslant 0$ for all $n\geqslant  2$, and proved that this would imply the Riemann Hypothesis (see \cite{BorweinCRW}). Unfortunately, this supposition occurred to be far from the truth, but the smallest counter-example is $n=906150257$.

As in the original Erd\H {o}s Discrepancy Problem, one may consider a relaxed version of majority coloring with some parameter $h$. Let us call a coloring $c$ of $A_{s,k}$ an \emph{$h$-majority} coloring if it satisfies:
\begin{equation}
c(s)\sum_{j=1}^{k}c(js)\leqslant h.
\end{equation}
This leads to the following conjecture.

\begin{conjecture}\label{Conjecture Merdos}
	There exists a constant $h$ and a red-blue coloring of $\mathbb N$ which is $h$-majority on every homogeneous arithmetic progression $A_{s,k}$.
\end{conjecture}
A natural first attempt is to look for an appropriate multiplicative coloring, which leads to the following problem.
\begin{conjecture}\label{Conjecture Merdos Multiplicative}
	There exists a constant $h$ and a completely multiplicative function $c$ satisfying 
	\begin{equation}
	\sum_{j=1}^{k}c(j)\leqslant h,
	\end{equation}
	for all $k\geqslant1$.
\end{conjecture}
We made some computer experiments with functions that are close to the Liouville function $\lambda(n)$ in the sense that only a few small primes have sign $+1$. For instance, switching only the sign of one small prime gives usually a function with much smaller partial sums. However, a standard argument using the Riemann Zeta function shows that in order to get a function satisfying Conjecture \ref{Conjecture Merdos Multiplicative} one has to switch signs of infinitely many primes.

\section{Final remarks}
Curiously, our initial impulse for Theorem \ref{Theorem Main} came from a different direction and was related to the following problem posed independently by Pach and P\'{a}lv\"{o}lgyi (see \cite{CaicedoCP}).
\begin{conjecture}\label{Conjecture Rainbow}
	For every $k\in \mathbb{N}$ there is a $k$-coloring of $\mathbb N$ such that every set $A_{s,k}$ is \emph{rainbow}.
\end{conjecture}
 Notice that the above statement easily implies the assertion of Theorem \ref{Theorem Main}. Indeed, a desired red-blue coloring can be obtained by splitting the set of $k$ colors into two subsets of (almost) the same cardinality.

Another consequence of Conjecture \ref{Conjecture Rainbow} is a positive answer to the following question of Graham \cite{Graham}: Is it true that among any $n$ distinct positive integers $a_1,a_2,\dots,a_n$ there is always a pair $a_i,a_j$ satisfying $\frac{a_i}{\gcd (a_i,a_j)}\geqslant n$? The problem was solved in the affirmative for sufficiently large $n$ by Szegedy \cite{Szegedy} and independently by Zaharescu \cite{Zaharescu}. Then Balasubramanian and Soundararajan \cite{BalasubramanianS} gave a complete solution by using methods of Analytic Number Theory.

To see a connection between these two problems, consider a graph $G_k$ on positive integers in which two numbers $r,s$ are joined by an edge if and only if $\frac{r}{\gcd(r,s)}\leqslant k$ and $\frac{s}{\gcd(r,s)}\leqslant k$. Let $\omega(G_k)$ and $\chi(G_k)$ denote the clique number and the chromatic number of the graph $G_k$, respectively. Then Graham's problem is equivalent to $\omega(G_k)=k$, while Conjecture \ref{Conjecture Rainbow} is equivalent to a much stronger statement that $\chi(G_k)=k$ (see \cite{BosekDGSSZ DM}, \cite{CaicedoCP}).

Going back to the Erd\H {o}s Discrepancy Problem, it is natural to wonder to what extent the original question can have a positive answer. Let us call a pair of sets $(S,K)$ \emph{cute} if there is a constant $h$ and a red-blue coloring of $\mathbb N$ such that every set $A_{s,k}$ is $h$-balanced, for all $s\in S$ and $k\in K$. So, the result of Tao \cite{Tao} says that the pair $(\mathbb{N},\mathbb{N})$ is not cute, while Theorem \ref{Theorem Main} asserts that $(\mathbb{N},K)$ is cute for every singleton $K=\{k\}$. It is not hard to prove that there are infinite sets $K$ for which $(\mathbb{N},K)$ is still cute. For instance, one may use the multiplicative coloring $b$ from Lemma \ref{Lemma Borwein} to infer that $(\mathbb{N},K)$ is cute if $K$ is the set of positive integers avoiding $1$'s in their base $3$ expansion. Notice however, that this set $K$ has density zero. Is there a cute pair $(\mathbb{N},K)$ with $K$ of positive density?

On the other hand, there exist dense sets $S$ for which the pair $(S,\mathbb N)$ is cute. For instance, the alternating red-blue coloring of $\mathbb{N}$ is perfectly balanced on all sets $A_{s,k}$ with $s$ odd (see \cite{LeongShallit}). Is there a cute pair $(S,K)$ with both sets of density one?

\end{document}